\documentclass[12pt]{article}
\usepackage{amssymb}
\usepackage{mathrsfs}
\usepackage[hypertex]{hyperref}
\usepackage{amsfonts}
\usepackage{graphicx}
\usepackage{mathptmx}
\usepackage{latexsym,amsmath,amssymb,amsfonts,amsthm}

\newcommand{\bcen}{\begin{center}}
\newcommand{\ecen}{\end{center}}
\newtheorem{theorem}{Theorem}[section]
\newtheorem{lemma}[theorem]{Lemma}

\newtheorem{remark}[theorem]{Remark}

\begin{document}
\setcounter{page}{1}
\title{A stability result for translating space-like graphs in Lorentz manifolds}
\author{Ya Gao, ~~~~~Jing Mao$^{\dag}$,~~~~~Chuanxi Wu}

\date{}
\protect \footnotetext{\!\!\!\!\!\!\!\!\!\!\!\!{~$^{\dag}$Corresponding author}
 }
\maketitle ~~~\\[-15mm]

\begin{center}
{\footnotesize Faculty of Mathematics and Statistics,\\
 Key Laboratory of Applied
Mathematics of Hubei Province, \\
Hubei University, Wuhan 430062, China\\
 Email:  jiner120@163.com  }
\end{center}

%\\[5mm]

\begin{abstract}
In this paper, we investigate space-like graphs defined over a
domain $\Omega\subset M^{n}$ in the Lorentz manifold
$M^{n}\times\mathbb{R}$ with the metric $-ds^{2}+\sigma$, where
$M^{n}$ is a complete Riemannian $n$-manifold with the metric
$\sigma$, $\Omega$ has piecewise smooth boundary, and $\mathbb{R}$
denotes the Euclidean $1$-space. We can prove an interesting
stability result for translating space-like graphs in
$M^{n}\times\mathbb{R}$ under a conformal transformation.
 \end{abstract}

\markright{\sl\hfill Y. Gao, J. Mao, C.-X. Wu\hfill}

\section{Introduction}
\renewcommand{\thesection}{\arabic{section}}
\renewcommand{\theequation}{\thesection.\arabic{equation}}
\setcounter{equation}{0}

Recent years, the study of submanifolds of constant curvature in
product manifolds attracts many geometers' attention. For instance,
Hopf in 1955 discovered that the complexification of the traceless
part of the second fundamental form of an immersed surface
$\mathcal{U}^{2}$, with CMC $H$, in $\mathbb{R}^3$ is a holomorphic
quadratic differential $Q$ on $\mathcal{U}^{2}$, and then he used
this observation to get his well-known conclusion that any immersed
CMC sphere $\mathbb{S}^{2}\hookrightarrow\mathbb{R}^3$ is a standard
distance sphere with radius $1/H$. By introducing a generalized
quadratic differential $\widetilde{Q}$ for immersed surfaces
$\mathcal{U}^{2}$ in product spaces $\mathbb{S}^{2}\times\mathbb{R}$
and $\mathbb{H}^{2}\times\mathbb{R}$, with $\mathbb{S}^{2}$,
$\mathbb{H}^{2}$ the $2$-dimensional sphere and hyperbolic surface
respectively, Abresch and Rosenberg \cite{ar} can extend Hopf's
result to CMC spheres in these target spaces. Meeks and Rosenberg
\cite{mr} successfully classified stable properly embedded
orientable minimal surfaces in the product space
$N\times\mathbb{R}$, where $N$ is a closed orientable Riemannian
surface. In fact, they proved that such a surface must be a product
of a stable embedded geodesic on $N$ with $\mathbb{R}$, a minimal
graph over a region of $N$ bounded by stable geodesics,
$N\times\{t\}$ for some $t\in\mathbb{R}$, or is in a moduli space of
periodic multigraphs parameterized by $P\times\mathbb{R}^{+}$, where
$P$ is the set of primitive (non-multiple) homology classes in
$H_{1}(N)$. Mazet, Rodr\'{\i}guez and Rosenberg \cite{lmh} analyzed
properties of periodic minimal or CMC surfaces in the product
manifold $\mathbb{H}^{2}\times\mathbb{R}$, and they also construct
examples of periodic minimal surfaces in
$\mathbb{H}^{2}\times\mathbb{R}$. In \cite{hfj}, Rosenberg, Schulze
and Spruck showed that a properly immersed minimal hypersurface in
$N\times\mathbb{R}^{+}$ equals some slice $N\times\{c\}$ when $N$ is
a complete, recurrent $n$-dimensional Riemannian manifold with
bounded curvature. Very recently, Gao, Mao and Song \cite{gms}
proved the existence and uniqueness of solutions to the CMC equation
with nonzero Neumann boundary data in product manifold
$N^{n}\times\mathbb{R}$, where $N^{n}$ is an $n$-dimensional
($n\geq2$) complete Riemannian manifold with nonnegative Ricci
curvature. Equivalently, this conclusion gives the existence of CMC
graphic hypersurfaces defined over a compact strictly convex domain
$\Omega\subset N^{n}$ and having nonvanishing contact angle. Of
course, for more information, readers can check references therein
of these papers. Hence, it is interesting and important to consider
submanifolds of constant curvature in the product manifold of type
$N^{n}\times\mathbb{R}$.

Inspired by Shahriyari's progress on complete translating graphs in
$\mathbb{R}^{3}$ (see \cite{ls} for details) and the Jenkins-Serrin
theory on minimal graphs and CMC graphs, Zhou \cite{hyz} considered
complete translating, minimal and CMC graphs in $3$-dimensional
product manifold $N^{2}\times\mathbb{R}$ over a domain
$\Omega\subset N^{2}$, where $N^{2}$ is a complete Riemannian
surface, and successfully showed the boundary behavior of $\Omega$.
This conclusion extends some of Shahriyari's conclusions in
\cite{ls} from the Euclidean $3$-space $\mathbb{R}^{3}$ to the
setting of $3$-dimensional product space $N^{2}\times\mathbb{R}$.

\emph{Stability} plays an important role in the study of minimal or
CMC hypersurfaces in Euclidean space or, more generally, product
manifolds. For instance, if stability assumption was made, nice
curvature estimates or classification results for minimal or CMC
surfaces can be obtained -- see, e.g., \cite{cm1,cm2,mr,rs,srz,hyz}.

The famous Bernstein theorem (holds only for $n\leq7$) in the
Euclidean space says that the entire nonparametric minimal
hypersurfaces in $\mathbb{R}^{n+1}$, $n\leq7$, are hyperplanes (see
\cite{ssy}). Calabi \cite{ec} (for $n\leq4$), Cheng-Yau \cite{cy}
(for all $n$) proved that a complete maximal spacelike hypersurface
in the flat Lorentz-Minkowski $(n+1)$-space
$\mathbb{L}^{n+1}\equiv\mathbb{R}^{n}_{1}$ is totally geodesic.
Therefore, specially, the only entire nonparametric maximal
space-like hypersurfaces in $\mathbb{R}^{n}_{1}$
 are space-like hyperplanes. This interesting example shows that it
 is meaningful to ask whether classical results in Riemannian
 geometry (or specially the Euclidean space) can be transplanted
 to pesudo-Riemannian geometry (or specially the pesudo-Euclidean
 space) or not. This example also shows that, in some aspect, there
 exists essential difference between the Euclidean space and the
pesudo-Euclidean
 space.

Motivated by the previous experience, we try to get stability
conclusions in Lorentz manifolds of type $M^{n}\times\mathbb{R}$.
Fortunately, so far, we get one -- see Theorem \ref{THEOREM3.1} for
details. In order to state our conclusion clearly, we need to
introduce some notions first.

Throughout this paper, denote by $M^{n}\times\mathbb{R}$, with the
metric $-ds^{2}+\sigma$, an $(n+1)$-dimensional ($n\geq2$) Lorentz
manifold
 where $M^{n}$ is a complete Riemannian $n$-manifold with the metric
 $\sigma$. For a domain $\Omega\subset M^{n}$ with piecewise smooth boundary, a \emph{translating space-like graph} in the Lorentz
 $(n+1)$-manifold $M^{n}\times\mathbb{R}$ is the space-like graph of
 $u(x)$, where $u(x):\Omega\rightarrow\mathbb{R}$ is a solution of
 the following mean curvature type equation
\begin{eqnarray} \label{trg}
\mathrm{div}\left(\frac{Du}{\sqrt{1-|Du|^{2}}}\right)=\frac{c}{\sqrt{1-|Du|^{2}}},
\end{eqnarray}
where $D$ is a covariant derivative operator on $M^{n}$,
$\mathrm{div}(\cdot)$ denotes the divergence operator, and $c$ is a
constant. Translating space-like graphs by mean curvature flow (MCF
for short) in the Lorentz manifold $M^{n}\times\mathbb{R}$ are
translating surfaces that can be viewed as a space-like graph of a
function over a domain. In fact, let $\{x,u(x)\}$ be a space-like
graphic surface defined over $\Omega\subset M^{n}$ in the Lorentz
manifold $M^{n}\times\mathbb{R}$, and then, since the mean curvature
of the space-like surface is (see \cite[Sect. 1]{chmx} in Section
\ref{sect3} here for this calculation)
\begin{eqnarray*}
H=\mathrm{div}\left(\frac{Du}{\sqrt{1-|Du|^{2}}}\right),
\end{eqnarray*}
the graph of $u$ is a vertically translating space-like with
constant speed $c$ if and only if $u$ is a solution to the equation
(\ref{trg}). Recently, Mao and his collaborators \cite{chmx} showed
that along the nonparametric MCF with prescribed contact angle
boundary condition in the Lorentz $3$-manifold
$M^{2}\times\mathbb{R}$, if $M^{2}$ has nonnegative Gaussian
curvature, then
 the evolution of space-like graphs over
compact strictly convex domains in $M^{2}$ exists for all the time
and solutions of the flow converge to ones moving only by
translation. Translating solutions play an important role in the
study of type-II singularities of the MCF. For instance, Angenent
and Vel\'{a}zquez \cite{av1,av2} gave some examples of convergence
which implies that type-II singularities of the MCF there are
modeled by translating surfaces.

Denote by $\widetilde{M^{n}\times\mathbb{R}}$ the
$(n+1)$-dimensional pseudo-Riemannian manifold
\begin{eqnarray*}
\{(x,s)|x\in M^{n},s\in\mathbb{R}\}
\end{eqnarray*}
equipped with the weighted metric
$e^{cs}(-ds^{2}+\sigma_{ij}dx^{i}dx^{j})$. Clearly,
$\widetilde{M^{n}\times\mathbb{R}}$ can be achieved by the Lorentz
 $(n+1)$-manifold $M^{n}\times\mathbb{R}$ with a conformal
 transformation to its Lorentzian metric. Here, we have used Einstein summation convention,
 that is, summation should be done to repeated subscripts and
 superscripts. In the sequel, without specification, Einstein summation
 convention will be always used. We can prove a stability result for
 translating space-like graphs as follows:

\begin{theorem} \label{THEOREM3.1}
Assume that $u(x)$ is a solution to (\ref{trg}). Then
$\Sigma=\{x,u(x))|x\in\Omega\}$ is a stable, maximal space-like
graph in $\widetilde{M^{n}\times\mathbb{R}}$.
\end{theorem}

The paper is organized as follows. In Section \ref{sect2}, some
useful formulas for space-like hypersurfaces in a Lorentz manifold
will be recalled. The proof of Theorem \ref{THEOREM3.1} will be
given in Section \ref{sect3}. Meanwhile, as a byproduct, a
convergence result related to maximal, CMC or translating space-like
graphs in Lorentz manifolds will also be shown. In Section
\ref{sect4}, examples related to the existence of translating
space-like graphs in the Lorentz
 $(n+1)$-manifold $M^{n}\times\mathbb{R}$ will be introduced.

\section{Geometry of space-like hypersurfaces in a Lorentz manifold}
\renewcommand{\thesection}{\arabic{section}}
\renewcommand{\theequation}{\thesection.\arabic{equation}}
\setcounter{equation}{0} \label{sect2}

Given an $(n+1)$-dimensional Lorentz manifold
$\left(\overline{M}^{n+1},\overline{g}\right)$, with the metric
$\overline{g}$, and its space-like hypersurface $M^{n}$. For any
$p\in M^{n}$, one can choose a local  Lorentzian orthonormal frame
field $\{e_{0},e_{1},e_{2},\ldots,e_{n}\}$ around $p$ such that,
restricted to $M^{n}$, $e_{1},e_{2},\ldots,e_{n}$ form orthonormal
frames tangent to $M^{n}$. Taking the dual coframe field
$\{w_{0},w_{1},w_{2},\ldots,w_{n}\}$ such that the Lorentzian metric
$\overline{g}$ can be written as
$\overline{g}=-w_{0}^{2}+\sum_{i=1}^{n}w_{i}^{2}$. Making the
convention on the range of indices
\begin{eqnarray*}
0\leq\alpha,\beta,\gamma,\ldots\leq n; \qquad\qquad 1\leq
i,j,k\ldots\leq n,
\end{eqnarray*}
and doing differentials to forms $w_{\alpha}$, one can easily get
the following structure equations
\begin{eqnarray}
&&(\mathrm{Gauss~ equation})\qquad \qquad R_{ijkl}=\overline{R}_{ijkl}-(h_{ik}h_{jl}-h_{il}h_{jk}), \label{Gauss}\\
&&(\mathrm{Codazzi~ equation})\qquad \qquad h_{ij,k}-h_{ik,j}=\overline{R}_{0ijk},  \label{Codazzi}\\
&&(\mathrm{Ricci~ identity})\qquad \qquad
h_{ij,kl}-h_{ij,lk}=\sum\limits_{m=1}^{n}h_{mj}R_{mikl}+\sum\limits_{m=1}^{n}h_{im}R_{mjkl},
\label{Ricci}
\end{eqnarray}
and the Laplacian of the second fundamental form $h_{ij}$ of $M^{n}$
as follows
\begin{eqnarray} \label{LF}
&&\Delta
h_{ij}=\sum\limits_{k=1}^{n}\left(h_{kk,ij}+\overline{R}_{0kik,j}+\overline{R}_{0ijk,k}\right)+
\sum\limits_{k=1}^{n}\left(h_{kk}\overline{R}_{0ij0}+h_{ij}\overline{R}_{0k0k}\right)+\nonumber\\
&&\qquad\qquad
\sum\limits_{m,k=1}^{n}\left(h_{mj}\overline{R}_{mkik}+2h_{mk}\overline{R}_{mijk}+h_{mi}\overline{R}_{mkjk}\right)\nonumber\\
&&\qquad\quad
-\sum\limits_{m,k=1}^{n}\left(h_{mi}h_{mj}h_{kk}+h_{km}h_{mj}h_{ik}-h_{km}h_{mk}h_{ij}-h_{mi}h_{mk}h_{kj}\right),
\end{eqnarray}
where $R$ and $\overline{R}$ are the curvature tensors of $M^{n}$
and $\overline{M}^{n+1}$ respectively, $A:=h_{ij}w_{i}w_{j}$ is the
second fundamental form with $h_{ij}$ the coefficient components of
the tensor $A$, $\Delta$ is the Laplacian on the hypersurface
$M^{n}$, and, as usual, the comma ``," in subscript of a given
tensor means doing covariant derivatives -- this convention will
also be
 used in the sequel. For detailed derivation of the above
formulae, we refer readers to, e.g., \cite[Section 2]{hzl}.

Clearly, in our setting here, all formulas mentioned in this section
can be used directly with
$\overline{M}^{n+1}=M^{n}\times\mathbb{R}$.

\section{Stability}
\renewcommand{\thesection}{\arabic{section}}
\renewcommand{\theequation}{\thesection.\arabic{equation}}
\setcounter{equation}{0} \label{sect3}

Similar to the calculation in \cite[Sect. 1]{chmx}, for the
space-like graph $\Sigma=\{(x,u(x))|x\in\Omega\}$, defined over
$\Omega\subset M^{n}$, in the Lorentz $(n+1)$-manifold
$M^{n}\times\mathbb{R}$ with the metric
$\overline{g}:=\sigma_{ij}dw^{i}\otimes dw^{j}-ds\otimes ds$,
tangent vectors are given by
\begin{eqnarray*}
X_{i}=\partial_{i}+D_{i}u\partial_{s}, \qquad i=1,2,\ldots,n,
\end{eqnarray*}
and the corresponding upward unit normal vector is given by
\begin{eqnarray*}
\vec{v}=\frac{1}{\sqrt{1-|Du|^2}}\left(\partial_s+D^{j}
u\partial_{j}\right),
\end{eqnarray*}
where $D^{j}u=\sigma^{ij}D_{i}u$. Denote by $\overline{\nabla}$ the
gradient operator on $M^{n}\times\mathbb{R}$, and then the second
fundamental form $h_{ij}dw^{i}\otimes dw^{j}$ of $\Sigma$ is given
by
\begin{eqnarray*}
h_{ij}=-\langle\overline{\nabla}_{X_i}
X_j,\vec{v}\rangle=\frac{1}{\sqrt{1-|Du|^2}}D_{i}D_{j}u.
\end{eqnarray*}
Moreover, the scalar mean curvature of $\Sigma$ is
\begin{eqnarray} \label{hf}
\qquad
H=\sum_{i=1}^{n}h^i_i=\frac{1}{\sqrt{1-|Du|^2}}\left(\sum_{i,k=1}^{n}g^{ik}D_{k}D_{i}u\right)&=&
\frac{\sum_{i,k=1}^{n}\left(\sigma^{ik}+\frac{D^{i}uD^{k}u}{1-|Du|^{2}}\right)D_{k}D_{i}u}{\sqrt{1-|Du|^2}}\nonumber\\
&=&\mathrm{div}\left(\frac{Du}{\sqrt{1-|Du|^{2}}}\right),
\end{eqnarray}
where $g^{ik}$ is the inverse of the induced Riemannian metric $g$
on the space-like graph $\Sigma$.
 Denote by $\Theta$ the angle function of $\Sigma$, and then
using (\ref{trg}), the above equality can be written equivalently as
\begin{eqnarray} \label{c-1}
H=-c\Theta=-c\langle\vec{v},\partial_{s}\rangle.
\end{eqnarray}

\begin{proof} [Proof of Theorem \ref{THEOREM3.1}]
The area functional of $\widetilde{M^{n}\times\mathbb{R}}$ is given
by
\begin{eqnarray*}
F(\Sigma)=\int_{\Sigma}e^{cs}d\mu,
\end{eqnarray*}
where $d\mu$ is the volume element of $\Sigma$ induced by the metric
$g$ of the Lorentz $(n+1)$-manifold $M^{n}\times\mathbb{R}$. Let
$\Sigma_{r}$ be a family of surfaces satisfying
\begin{eqnarray} \label{cf}
\frac{\partial\Sigma_{r}}{\partial
r}\Bigg{|}_{r=0}=\phi\vec{v}~~~~\mathrm{with}~~~~\Sigma_{0}=\Sigma,
\end{eqnarray}
where $\phi(x)$ is a smooth function defined on $\Sigma$ with
compact support. Treating $\Sigma_{r}$ as a curvature flow of
$\Sigma$ in the Lorentz $(n+1)$-manifold $M^{n}\times\mathbb{R}$,
and by direct calculation, it follows that:

\begin{lemma} \label{lemma3.1}
Along the curvature flow (\ref{cf}), we have
\begin{eqnarray} \label{c-2}
\begin{split}
&\frac{\partial\vec{v}}{\partial r}\Bigg{|}_{r=0}=\nabla\phi,\\
&\frac{\partial H}{\partial
r}\Bigg{|}_{r=0}=\Delta\phi-\left(|A|^{2}+\overline{\mathrm{Ric}}(\vec{v},\vec{v})\right)\phi,
\end{split}
\end{eqnarray}
where, following the convention used in Section \ref{sect2},
$\nabla$ and $\Delta$ denote the covariant derivative and the
Laplacian of $\Sigma$ respectively, and
$\overline{\mathrm{Ric}}(\cdot,\cdot)$ stands for the Ricci tensor
of the ambient space $M^{n}\times\mathbb{R}$.
\end{lemma}

\begin{proof}
First, we have
\begin{eqnarray*}
\begin{split}
\frac{\partial\vec{v}}{\partial r}\Bigg{|}_{r=0}&=\left\langle\frac{\partial\vec{v}}{\partial{r}},(\Sigma_{r})_{,i}\right\rangle g^{ik}(\Sigma_{r})_{,k}\Bigg{|}_{r=0}\\
&=
-\left\langle\vec{v},(\phi\vec{v})_{,i}\right\rangle g^{ik}(\Sigma_{r})_{,k}\Bigg{|}_{r=0}\\
&=\phi_{,i}g^{ik}(\Sigma_{r})_{,k}\Bigg{|}_{r=0}=\nabla\phi,
\end{split}
\end{eqnarray*}
where, following the convention used in Section \ref{sect2},
$(\cdot)_{,k}$ means doing covariant derivative with respect to the
tangent vector $X_{k}$ on the translating space-like graph $\Sigma$.

Second, we have
\begin{eqnarray*}
\begin{split}
\frac{\partial g_{lm}}{\partial r}\Bigg{|}_{r=0}&=\frac{\partial}{\partial r}\left\langle(\Sigma_{r})_{,l},(\Sigma_{r})_{,m}\right\rangle\Bigg{|}_{r=0}\\
&=
2\left\langle(\phi\vec{v})_{,l},(\Sigma_{r})_{,m}\right\rangle\Bigg{|}_{r=0}\\
&= -2\phi\left\langle\vec{v},(\Sigma_{r})_{,lm}\right\rangle\Bigg{|}_{r=0}=2\phi h_{ij},
\end{split}
\end{eqnarray*}
and
\begin{eqnarray*}
\begin{split}
\frac{\partial h_{ij}}{\partial r}\Bigg{|}_{r=0}&=-\frac{\partial}{\partial r}\left\langle\vec{v},(\Sigma_{r})_{,ij}\right\rangle\Bigg{|}_{r=0}\\
&= -\left\langle\phi_{,l}(\Sigma_{r})_{,m}g^{ml},(\Sigma_{r})_{,ij}\right\rangle\Bigg{|}_{r=0}
-\left\langle\vec{v},(\phi\vec{v})_{,ij}\right\rangle\Bigg{|}_{r=0}\\
&=
-\left\langle\phi_{,l}(\Sigma_{r})_{,m}g^{ml},\Gamma_{ij}^{k}(\Sigma_{r})_{,k}+h_{ij}\vec{v}\right\rangle\Bigg{|}_{r=0}-\left\langle\vec{v},\left(\phi_{,j}\vec{v}+\phi h_{jl}g^{lm}(\Sigma_{r})_{,m}\right)_{,i}\right\rangle\Bigg{|}_{r=0}\\
&=
-\Gamma_{ij}^{k}\phi_{,k}+\phi_{,ij}+\phi h_{jl}g^{lm}h_{im}\\
&= \nabla_{i}\nabla_{j}\phi+\phi h_{il}g^{lm}h_{im},
\end{split}
\end{eqnarray*}
where, as usual, $\Gamma^{k}_{ij}$ denote Christoffel symbols
determined by the metric $g$.
 By \eqref{Gauss}, \eqref{Codazzi}, \eqref{LF} and
Simon's identity of $\phi$, we have
\begin{eqnarray*}
g^{ij}\frac{\partial h_{ij}}{\partial
r}\Bigg{|}_{r=0}=\Delta\phi+|A|^{2}\phi-\overline{\mathrm{Ric}}\left(\vec{v},\vec{v}\right)\phi,
\end{eqnarray*}
and then
\begin{eqnarray*}
\begin{split}
\frac{\partial H}{\partial r}\Bigg{|}_{r=0}&=\frac{\partial}{\partial r}\left(g^{ij}h_{ij}\right)\Bigg{|}_{r=0}\\
&=
-g^{il}\frac{\partial g_{lm}}{\partial r}|_{r=0}g^{mj}h_{ij}+g^{ij}\frac{\partial h_{ij}}{\partial r}|_{r=0}\\
&=
-2\phi|A|^{2}+\Delta\phi+|A|^{2}\phi-\overline{\mathrm{Ric}}\left(\vec{v},\vec{v}\right)\phi\\
&=
\Delta\phi-\left(|A|^{2}+\overline{\mathrm{Ric}}\left(\vec{v},\vec{v}\right)\right)\phi.
\end{split}
\end{eqnarray*}
This completes the proof of Lemma \ref{lemma3.1}.
\end{proof}

By \eqref{c-1} and \eqref{c-2}, it is not hard to obtain
\begin{eqnarray} \label{c-3}
\begin{split}
&\frac{\partial F(\Sigma_{r})}{\partial r}\Bigg{|}_{r=0}=\int_{\Sigma}\phi\left(H+c\left\langle\vec{v},\partial_{s}\right\rangle\right)e^{cs}d\mu=0,\\
&\frac{\partial^{2}F(\Sigma_{r})}{\partial^{2}r}\Bigg{|}_{r=0}=\int_{\Sigma}\phi\left[\Delta\phi-\left(|A|^{2}+\overline{\mathrm{Ric}}(\vec{v},\vec{v})\right)\phi+c\langle\nabla\phi,\partial_{s}\rangle\right]e^{cs}d\mu.
\end{split}
\end{eqnarray}
Define an elliptic operator $L$ as follows
\begin{eqnarray} \label{c-4}
L\phi=\Delta\phi-\left(|A|^{2}+\overline{\mathrm{Ric}}(\vec{v},\vec{v})\right)\phi+c\left\langle\nabla\phi,\partial_{s}\right\rangle.
\end{eqnarray}
Therefore, putting (\ref{c-4}) into the second equality of
(\ref{c-3}) yields
\begin{eqnarray} \label{c-5}
\frac{\partial^{2}F(\Sigma_{r})}{\partial^{2}r}\Bigg{|}_{r=0}=\int_{\Sigma}\phi
L\phi e^{cs}d\mu.
\end{eqnarray}
Now, we only need to show that the RHS of (\ref{c-5}) is
non-positive. Since $\Sigma$ is a space-like graph, its angle
function satisfies $\Theta=\langle\vec{v},\partial_{r}\rangle<0$.
Thus we can write $\phi=\eta\Theta$, where $\eta$ is another
function over $\Sigma$ with compact support. Then it follows that
\begin{eqnarray} \label{c-6}
\phi L\phi=\eta\Theta\left(\eta
L\Theta+\Theta\Delta\eta+2\langle\nabla\eta,\nabla\Theta\rangle+c\Theta\langle\nabla\eta,\partial_{s}\rangle\right).
\end{eqnarray}
The reason why we adapt this form is based on a general formula of
$\Delta\Theta$ as follows.

\begin{lemma}\label{LEMMA3.2}
For any $C^{2}$ space-like hypersurface $S$ in the Lorentz
$(n+1)$-manifold $M^{n}\times\mathbb{R}$, it holds that
\begin{eqnarray} \label{c-7}
\Delta\Theta-\left(|A|^{2}+\overline{\mathrm{Ric}}(\vec{v},\vec{v})\right)\Theta-\langle\nabla
H,\partial_{s}\rangle=0,
\end{eqnarray}
where $A$ is the second fundamental form of $S$.
\end{lemma}

\begin{proof}
Fix a point $p\in S$. Suitably choose an orthonormal frame field
$\{e_{1},e_{2},\ldots,e_{n}\}$ on $S$ such that
$\nabla_{e_{i}}e_{j}(p)=0$ and $\langle
e_{i},e_{j}\rangle=\delta_{ij}$. Then
$\overline{\nabla}_{e_{i}}e_{j}(p)=h_{ij}\vec{v}$, where, following
the convention used in Section \ref{sect2}, $\overline{\nabla}$
denotes the covariant derivative of the ambient space
$M^{n}\times\mathbb{R}$ and $\vec{v}$ is the unit normal vector of
$S$. It is easy to know that for any smooth vector field $X$,
$\overline{\nabla}_{X}\partial_{s}=0$. By direct calculation, one
has
\begin{eqnarray} \label{c-8}
\Delta\Theta(p)&=\nabla_{e_{i}}\nabla_{e_{i}}\langle\partial_{s},\vec{v}\rangle-\nabla_{\nabla_{e_{i}}e_{i}}\Theta(p)\nonumber\\
&=
e_{i}\langle\partial_{s},h_{ik}e_{k}\rangle(p)\nonumber\\
&= h_{ik,i}\langle\partial_{s},e_{k}\rangle+|A|^{2}\Theta.
\end{eqnarray}
Using the Codazzi equation (\ref{Codazzi}) directly yields
\begin{eqnarray*} \label{c-9}
h_{ik,i}=h_{ii,k}+\overline{R}_{0iki}.
\end{eqnarray*}
Hence, it gives
\begin{eqnarray} \label{c-10}
h_{ik,i}\langle\partial_{s},e_{k}\rangle=\langle\nabla
H,\partial_{s}\rangle+\overline{\mathrm{Ric}}\left(\vec{v},\langle\partial_{s},e_{k}\rangle
e_{k}\right).
\end{eqnarray}
Since $\overline{\nabla}_{X}\partial s=0$ for any vector $X$, we
know
\begin{eqnarray*} \label{c-11}
\langle\partial_{s},e_{k}\rangle e_{k}=\partial_{s}+\Theta\vec{v}
\end{eqnarray*}
and $\overline{\mathrm{Ric}}(\vec{v},\partial_{s})=0$. Putting these
two facts into (\ref{c-10}) implies
\begin{eqnarray*}
h_{ik,i}\langle\partial_{s},e_{k}\rangle=\langle\nabla
H,\partial_{s}\rangle+\overline{\mathrm{Ric}}(\vec{v},\vec{v})\Theta.
\end{eqnarray*}
The assertion of this lemma follows by combing the above equality
with \eqref{c-8} directly.
\end{proof}

Let us go back to the proof of Theorem \ref{THEOREM3.1}.  Since
$\Sigma$ is a translating space-like graph in the Lorentz
$(n+1)$-manifold $M^{n}\times\mathbb{R}$, one has $H=-c\Theta$, and then
\eqref{c-7} can be rewritten as
\begin{eqnarray*} \label{c-12}
L\Theta=0.
\end{eqnarray*}
Therefore \eqref{c-6} becomes
\begin{eqnarray*} \label{c-13}
\phi
L\phi=\eta\Theta\left(\Theta\Delta\eta+2\langle\nabla\eta,\nabla\Theta\rangle+c\Theta\langle\nabla\eta,\partial_{s}\rangle\right).
\end{eqnarray*}
On the other hand, the divergence of $\eta\Theta^{2}\nabla\eta
e^{cs}$ is
\begin{eqnarray} \label{c-14}
\begin{split}
\mathrm{div}\left(\eta\Theta^{2}\nabla\eta e^{cs}\right)&=\eta\Theta e^{cs}\left(\Theta\Delta\eta+2\langle\nabla\eta,\nabla\Theta\rangle+c\Theta\langle\nabla\eta,\partial_{s}\rangle\right)+\Theta^{2}|\nabla\eta|^{2}e^{cs}\\
&= \phi e^{cs}L\phi+\Theta^{2}|\nabla\eta|^{2}e^{cs}.
\end{split}
\end{eqnarray}
Combining (\ref{c-14}) with \eqref{c-5} and applying the divergence
theorem result in
\begin{eqnarray*}
\frac{\partial^{2}F(\Sigma_{r})}{\partial^{2}r}\Bigg{|}_{r=0}=-\int_{\Sigma}\Theta^{2}|\nabla\eta|^{2}e^{cs}d\mu\leq0.
\end{eqnarray*}
Then we conclude that the translating space-like graph $\Sigma$ is
stable and maximal in $\widetilde{M^{n}\times\mathbb{R}}$.
\end{proof}

Applying Lemma \ref{LEMMA3.2}, we can obtain the following
interesting rigidity result.

\begin{theorem}\label{THEOREM3.3}
Let $\{\Sigma_{n}\}_{n=1}^{\infty}$ be a sequence of smooth
connected space-like graphs in the Lorentz $(n+1)$-manifold
$M^{n}\times\mathbb{R}$ with diameter $\varrho$ converging uniformly
to a connected space-like hypersurface $\Sigma$ in the $C^{2}$
sense. If all $\Sigma_{n}$ are translating space-like graphs in the
interior of $\Sigma$, the angle function $\Theta$ satisfies that
$\Theta<0$ or $\Theta\equiv0$. The conclusion is also true in the
case of maximal or CMC space-like graphs.
\end{theorem}

\begin{proof}
Without loss of generality, we assume $\Theta<0$. By continuity, we
know that  in the interior of all $\Sigma_{n}$, $|A|^{2}<\beta_{1}$
holds for some positive constant $\beta_{1}$ depending only on
$M^{n}$.

Now, first, we assume that $\Sigma_{n}$ are maximal or CMC
space-like graphs. Then $\nabla H\equiv0$. By Lemma \ref{LEMMA3.2},
we have
\begin{eqnarray} \label{c-15}
\Delta\Theta-\left(|A|^{2}+\overline{\mathrm{Ric}}(\vec{v},\vec{v})\right)\Theta=0
\end{eqnarray}
on all $\Sigma_{n}$. Since
\begin{eqnarray*}
\overline{\mathrm{Ric}}(\vec{v},\vec{v})=\frac{u_{,k}^{2}(\Gamma_{kk,i}^{i}+\Gamma_{kk}^{l}\Gamma_{il}^{i}-\Gamma_{ik,k}^{i}-\Gamma_{ik}^{l}\Gamma_{kl}^{i})}{1-|Du|^{2}},
\quad i,k,l=1,2,\ldots,n,
\end{eqnarray*}
there exists a positive constant $\beta_{2}$ only depending on
$M^{n}$ such that
$\overline{\mathrm{Ric}}(\vec{v},\vec{v})\leq\beta_{2}$ in the
interior of all $\Sigma_{n}$. By (\ref{c-15}) we have
$\Delta\Theta\geq\left(\beta_{1}+\beta_{2}\right)\Theta$ on all
$\Sigma_{n}$. Because $\Sigma$ is the $C^{2}$ uniform limit of
$\Sigma_{n}$ as $n\rightarrow\infty$, it follows that $\Theta\leq0$
and $\Delta\Theta\geq(\beta_{1}+\beta_{2})\Theta$. By the strong
maximum principle of second-order elliptic equations, we can obtain
that $\Theta\equiv0$ or $\Theta<0$ on $\Sigma$.

Second, assume that $\Sigma_{n}$ are translating space-like graphs.
Then $H\equiv-c\Theta$ by \eqref{c-1}. Similar argument gives
\begin{eqnarray*}
\Delta\Theta\geq(\beta_{1}+\beta_{2})\Theta-c\left\langle\nabla\Theta,\partial_{s}\right\rangle
\end{eqnarray*}
on all $\Sigma_{n}$. Based on the strong maximum principle and the
fact that $\Theta\leq0$ on $\Sigma$, we also have $\Theta\equiv0$ or
$\Theta<0$ on $\Sigma$.
\end{proof}

\section{Examples of translating space-like graphs}
\renewcommand{\thesection}{\arabic{section}}
\renewcommand{\theequation}{\thesection.\arabic{equation}}
\setcounter{equation}{0} \label{sect4}

In this section, we construct some examples of translating
space-like graphs to MCF when the hypersurface $M^{n}$ has a domain
with certain warped product structure.

Suppose that $M^{n}$ is an $n$-dimensional ($n\geq2$) complete
Riemannian manifold with a metric $\sigma$ containing a domain
$M_{0}^{n}$ equipped with the following coordinate system:
\begin{eqnarray} \label{f-1}
\left\{\theta=(\theta_{2},\theta_{3},\ldots,\theta_{n})\in
\mathbb{S}^{n-1},
r\in[0,r_{0})\right\}~~~~\mathrm{with}~~~~\sigma=dr^{2}+h^{2}(r)d\theta^{2},
\end{eqnarray}
where $d\theta^{2}$ is the round metric on the unit $(n-1)$-sphere
$\mathbb{S}^{n-1}$,
 $h(r)$ is a
positive function satisfying $h(0)=0$, $h'(0)=1$ with $h'(r)\neq0$
for all $r\in(0,r_{0})$.

Now, with the help of examples constructed below, we can somehow
show the existence of translating space-like graphs in the Lorentz
$(n+1)$-manifold $M^{n}\times\mathbb{R}$ with the structure
\eqref{f-1} and the metric $\overline{g}$.

\begin{theorem}\label{THEOREM6.1}
Let $M^{n}$ be a complete Riemannian $n$-manifold mentioned above.
Let $u(r):[0,r_{0})\rightarrow\mathbb{R}$ be a $C^{2}$ solution of
the following ordinary differential equation (ODE for short)
\begin{eqnarray} \label{f-2}
\frac{u_{rr}}{1-u_{r}^{2}}+(n-1)\frac{h'(r)}{h(r)}u_{r}=c,
\end{eqnarray}
with $u_{r}(0)=0$ for $r\in[0,r_{0})$ and $|u_{r}|<1$. Then
$\Sigma=(x,u(r))$ for $r\in[0,r_{0})$ is a translating space-like
graph in the Lorentz $(n+1)$-manifold $M^{n}\times\mathbb{R}$, where
$x=(r,\theta)\in M_{0}^{n}$ given by \eqref{f-1}. If $r_{0}=\infty$,
then $\Sigma$ is complete.
\end{theorem}

\begin{remark}\label{LEMARK6.2}
\rm{Clearly, \eqref{f-2} is a second-order ODE whose component of
the second-order derivative term does not degenerate under the
assumption $|u_{r}|<1$. The existence of its solution is obvious.}
\end{remark}

\begin{proof}
If $r_{0}=\infty$, then $M_{0}^{n}$ is simply connected and should
be a whole $M^{n}$. Thus $\Sigma$ is complete.

In the rest part, we show that $\Sigma$ is a translating space-like
graph. By \eqref{c-1}, we know that here it is sufficient to derive
the identity
\begin{eqnarray} \label{f-3}
H=-c\Theta,
\end{eqnarray}
where $H$ is the mean curvature of $\Sigma$ and $\vec{v}$ is its
upward normal vector.

Fix a point $(x,u(x))$ on $\Sigma$, where $x\in M_{0}^{n}$ and the
polar coordinate of $x$ in $M_{0}^{n}$ is not $(0,0,\ldots,0)$.
Clearly, the polar coordinate system on $M_{0}^{n}$ given by
\eqref{f-1} determines a frame field
$\{\partial_{r},\partial_{\theta_{2}},\ldots,\partial_{\theta_{n}}\}$
naturally. For the space-like graph $\Sigma$ determined by
$u(x)=u(r)$ in the Lorentz $(n+1)$-manifold $M^{n}\times\mathbb{R}$,
denote by $u_{r}$ and $u_{\theta_{i}}$, $i=2,3,\ldots,n$, the
partial derivatives of $u$. Since here $u(r)$ is a radial function,
$u_{\theta_{i}}\equiv0$, $i=2,3,\ldots,n$. Therefore, on $\Sigma$, a
natural frame $\{e_{1}=\partial r+u_{r}\partial_{s},
e_{i}=\partial_{\theta_{i}}\}$, $i=2,\ldots,n$ can be obtained,
where, as before, $\partial_{s}$ denotes the vector field tangent to
$\mathbb{R}$. Then the Riemannian metric on $\Sigma$ and the upward
unit normal vector of $\Sigma$ are given by
\begin{eqnarray*}
g_{11}=\left\langle
e_{1},e_{1}\right\rangle=1-u_{r}^{2},~~~~g_{kl}=g_{lk}=\left\langle
e_{l},e_{k}\right\rangle=0,~~~~k\neq l,
\end{eqnarray*}
\begin{eqnarray*}
g_{ii}=\left\langle e_{i},e_{i}\right\rangle=h^{2}(r),~~~~i=2,\ldots,n,
\end{eqnarray*}
and
\begin{eqnarray*}
\vec{v}=\frac{\partial_{s}+u_{r}\partial_{r}}{\sqrt{1-u_{r}^{2}}}.
\end{eqnarray*}
By direct calculation, its second fundamental forms are
\begin{eqnarray*}
h_{11}=-\left\langle\overline{\nabla}_{e_{1}}e_{1},\vec{v}\right\rangle=\frac{u_{rr}}{\sqrt{1-u_{r}^{2}}},
\end{eqnarray*}
and
\begin{eqnarray*}
h_{ii}=-\left\langle\overline{\nabla}_{e_{i}}e_{i},\vec{v}\right\rangle=-\left\langle-h(r)h'(r)\partial_{r},\vec{v}\right\rangle=\frac{h'(r)h(r)u_{r}}{\sqrt{1-u_{r}^{2}}},~~~~i=2,\ldots,n.
\end{eqnarray*}
where we use the fact
\begin{eqnarray*}
\left\langle\overline{\nabla}_{e_{i}}e_{i},\partial
r\right\rangle=-h'(r)h(r),~~~~i=2,\ldots,n.
\end{eqnarray*}
Then, by \eqref{f-2}, the mean curvature of $\Sigma$ with respect to
$\vec{v}$ is
\begin{eqnarray*}
H=g^{11}h_{11}+g^{22}h_{22}+\ldots+g^{nn}h_{nn}=\frac{1}{\sqrt{1-u_{r}^{2}}}\left(\frac{u_{rr}}{1-u_{r}^{2}}+(n-1)\frac{h'(r)}{h(r)}u_{r}\right)
=\frac{c}{\sqrt{1-u_{r}^{2}}}.
\end{eqnarray*}
 On the other hand, we have
\begin{eqnarray*}
\Theta=\left\langle\vec{v},\partial_{s}\right\rangle=
\left\langle\frac{\partial_{s}+u_{r}\partial_{r}}{\sqrt{1-u_{r}^{2}}},\partial_{s}\right\rangle=-\frac{1}{\sqrt{1-u_{r}^{2}}}.
\end{eqnarray*}
Hence in our case here we have $H=-c\Theta$, which implies $\Sigma$
is a translating space-like graph. The proof is finished.
\end{proof}

\section*{Acknowledgments}
\renewcommand{\thesection}{\arabic{section}}
\renewcommand{\theequation}{\thesection.\arabic{equation}}
\setcounter{equation}{0} \setcounter{maintheorem}{0}

This research was supported in part by the NSF of China (Grant Nos.
11801496 and 11926352), the Fok Ying-Tung Education Foundation
(China), and Hubei Key Laboratory of Applied Mathematics (Hubei
University).

\end{document}